\documentclass[12pt,a4paper]{amsart}

\usepackage[a4paper, total={6in, 23.7cm}]{geometry}
\usepackage{amsmath,amsthm,amsfonts,amssymb,mathrsfs,amsbsy}
\usepackage{graphicx}
\usepackage{color}
\usepackage[noadjust]{cite}
\usepackage{marginnote}
\usepackage{xifthen}
\usepackage{mathtools}
\usepackage{dsfont}
\usepackage{bbm}
\usepackage{appendix}
\usepackage{delarray}
\usepackage{enumerate}
\usepackage{enumitem}

\usepackage{xcolor}
\usepackage{tikz}
\usepackage[hypertexnames=false]{hyperref} 
\usepackage{autonum}

\binoppenalty=\maxdimen
\relpenalty=\maxdimen

\newenvironment{spm}
{\bigl(\begin{smallmatrix}}
	{\end{smallmatrix}\bigr)}

\newenvironment{nalign}
{\begin{equation}\begin{aligned}}
		{\end{aligned}\end{equation}\ignorespacesafterend}

\newtheorem{theorem}{Theorem}

\newtheorem{prop}[theorem]{Proposition}

\numberwithin{equation}{section}
\numberwithin{theorem}{section}
\numberwithin{figure}{section}

\theoremstyle{remark}
\newtheorem{rem}[theorem]{Remark}

\theoremstyle{remark}

\theoremstyle{theorem}
\newtheorem{assumption}[theorem]{Assumption}

\newcommand\R{\mathbb{R}}

\newcommand{\oU}{\overline{U}}
\newcommand{\oV}{\overline{V}}
\newcommand{\rev}[1]{{\color{black} #1}}

\newcommand{\UV}{( U, V)}
\newcommand{\UVx}{\big( U(x), V(x)\big)}

\newcommand\V{\mathcal{V}}

\newcommand{\oo}{\overline{\Omega}}

\renewcommand\L{\mathcal{L}}

\newcommand{\W}[2]
{{
		W^{#1,#2}_
		{
			\ifthenelse
			{
				\equal{#1}{2}
			}
			{
				\nu
			}
			{
			}
		}(\Omega)
}}

\newcommand{\Li}{{L^\infty(\Omega)}}

\newcommand{\Lp}[1]
{{
		L^
		{
			\ifthenelse
			{
				\isempty{#1}
			}
			{
				p
			}
			{
				#1
			}
		}(\Omega)
}}

\definecolor{lime}{HTML}{A6CE39}
\DeclareRobustCommand{\orcidicon}{%
	\begin{tikzpicture}
		\draw[lime, fill=lime] (0,0) 
		circle [radius=0.16] 
		node[white] {{\fontfamily{qag}\selectfont \tiny ID}};
		\draw[white, fill=white] (-0.0625,0.095) 
		circle [radius=0.007];
	\end{tikzpicture}
	\hspace{-2mm}
}

\definecolor{lime}{HTML}{A6CE39}
\DeclareRobustCommand{\orcidicon}{%
	\begin{tikzpicture}
		\draw[lime, fill=lime] (0,0) 
		circle [radius=0.16] 
		node[white] {{\fontfamily{qag}\selectfont \tiny ID}};
		\draw[white, fill=white] (-0.0625,0.095) 
		circle [radius=0.007];
	\end{tikzpicture}
	\hspace{-2mm}
}

\begin{document} 

\title[Reaction-diffusion-ODE systems]{
Discontinuous stationary solutions \\ to certain reaction-diffusion systems}

\author[S. Cygan]{Szymon Cygan \href{https://orcid.org/0000-0002-8601-829X}{\orcidicon}}
\address[S. Cygan]{	Instytut Matematyczny, Uniwersytet Wroc\l{}awski, pl. Grunwaldzki 2/4, \hbox{50-384} Wroc\l{}aw, Poland \\ 
\href{https://orcid.org/0000-0002-8601-829X}{orcid.org/0000-0002-8601-829X}}
\email{szymon.cygan@math.uni.wroc.pl}
\urladdr {http://www.math.uni.wroc.pl/~scygan}

\author[A. Marciniak-Czochra]{Anna Marciniak-Czochra
\href{https://orcid.org/0000-0002-5831-6505}{\orcidicon}}
\address[A. Marciniak-Czochra]{
Institute of Applied Mathematics,  Interdisciplinary Center for Scientific Computing (IWR) and BIOQUANT, University of Heidelberg, 69120 Heidelberg, Germany\\
\href{https://orcid.org/0000-0002-5831-6505}{orcid.org/0000-0002-5831-6505}}
\email{anna.marciniak@iwr.uni-heidelberg.de}
\urladdr {http://www.biostruct.uni-hd.de/}

\author[G. Karch]{Grzegorz Karch \href{https://orcid.org/0000-0001-9390-5578}{\orcidicon}}
\address[G. Karch]{	
Instytut Matematyczny, Uniwersytet Wroc\l{}awski, pl. Grunwaldzki 2/4, \hbox{50-384} Wroc\l{}aw, Poland \\ 
\href{https://orcid.org/0000-0001-9390-5578}{orcid.org/0000-0001-9390-5578}}
\email{grzegorz.karch@math.uni.wroc.pl}
\urladdr {http://www.math.uni.wroc.pl/~karch}

\date{\today}

\dedicatory{Dedicated to Eiji Yanagida on the occasion of his birthday}



\begin{abstract}
	Systems consisting of a single ordinary differential equation coupled with one reaction-diffusion equation in a bounded domain and with \rev{the} Neumann boundary conditions \rev{are}  studied in the case of particular nonlinearities from the Brusselator model, the Gray-Scott model, the Oregonator model and a certain predator-prey model. It is shown that the considered systems have the both smooth and discontinuous stationary solutions, however, only discontinuous ones can be stable. 
\end{abstract}


\maketitle

\section{Introduction} 

We discuss   stationary solutions to certain particular systems consisting of a single ordinary differential equation coupled with one reaction-diffusion equation: 
\begin{nalign}
	u_t  &=   f(u,v), &&x\in\overline{\Omega}, \quad t>0,  \label{eq1}\\
	v_t  &=  \gamma \Delta v+g(u,v), &&  x\in \Omega, \quad t>0,
\end{nalign}
with   unknown functions $u = u (x,t)$ and $v = v(x,t)$. 
System \eqref{eq1} is considered in a bounded, open domain 
\rev{$\Omega\subset \R^N$} with \rev{the} $C^2$-boundary $\partial\Omega$.
Moreover,  the reaction-diffusion equation in \eqref{eq1} is 
supplemented with the Neumann boundary condition \noeqref{Neumann}
\begin{nalign}
	\label{Neumann}
	\partial_{\nu}v  =  0  \qquad \text{for}\quad x\in \partial\Omega, \quad t>0,
\end{nalign}
where $\partial_\nu = \nu\cdot\nabla$ with  $\nu$ denoting  the unit
outer normal vector to $\partial \Omega$. The letter $\gamma>0$ is a constant diffusion coefficient.
In the following, we denote by $\Delta_\nu$ the Laplacian operator together with the Neumann boundary condition \eqref{Neumann}.


Particular versions of such a system have been studied, {\it e.g.},~in the  works
\cite{MR3973251, MR3679890, MR3600397, MR3583499, MR3345329, MR3329327, MR3214197, MR3059757, MR3039206, MR2833346,MR3220545,MR684081,MR579554,10780947202173109,MR4213664,MR730252} where the following two types of stationary solutions $\UVx$ were identified:

\begin{enumerate}
	\item \textit{Regular stationary solutions}, where
	$ U(x) =  k\big(V(x)\big)$ for one $C^2$-function $ k$ and all $x\in \overline{\Omega}$.  
	\label{it:Type1}
	\item  \textit{Jump-discontinuous stationary solutions} with $ U(x)$ obtained using different branches of solutions to the equation $f(U,V)=0$ with respect to $ U$.
\end{enumerate} 

Regular  stationary solutions to  general problem \eqref{eq1}-\eqref{Neumann}
have been investigated recently in the paper \cite{CMCKS01} showing that \rev{they} all  
are unstable. The second work~\cite{CMCKS02}  is devoted to the second type stationary solutions with jump-discontinuities and provides  sufficient conditions for their existence and stability. 
The theory developed in both papers \cite{CMCKS01, CMCKS02}
can be  applied to several  reaction-diffusion-ODE models in spatially homogeneous environments arising from applications as {\it e.g.} studied in 
 \cite{MR3214197,MR3583499,MR3679890, MR3973251,MR3345329,MR2297947,MR684081,MR579554,perthame2020fast,MR2205561,gg}.

In this work, we explain how to apply general results from the papers \cite{CMCKS01, CMCKS02}  to systems of \rev{type} \eqref{eq1} with particular well-known nonlinearities which appear in Turing models. Recall that a system of reaction-diffusion equations 
\begin{nalign}\label{RD:gen} \rev{
	u_t = \varepsilon \Delta u + f(u,v), \qquad  v_t = \gamma \Delta v + g(u,v) }
\end{nalign}
(in a bounded domain and with the Neumann boundary condition) exhibits a Turing instability if it has a constant stationary solution which is stable if $\varepsilon = \gamma = 0$ and unstable for $\varepsilon>0$ and $\gamma>0$. 
In his work \cite{MR3363444}, 
Turing showed that this kind of instability may occurs in some particular reaction-diffusion systems with $\varepsilon>0$ sufficiently small. Due to different nonzero diffusion
rates, the system gives rise spontaneously to stationary pattern formation with a characteristic
length scale from an initial configuration (see \textit{e.g.}~\cite{MR3114654} and the references therein).

In this work, we apply \rev{the} theory from the papers \cite{CMCKS01,CMCKS02} to study nonconstant stationary solutions of some  reaction-diffusion models with the Turing instability in the diffusion degenerated case, namely,  when $\varepsilon = 0$ and $\gamma>0$.  
 In Section \ref{sec:Instability}, we consider the diffusion-degenerated  Gray-Scott model
\begin{nalign}
	u_t =  u^2 v - \alpha u, \qquad 
		v_t  = \gamma \Delta_\nu v - u ^2 v + \beta(1-v)
\end{nalign}
and the diffusion-degenerated  Brusselator system
\begin{nalign}
	u_t =  \alpha + u^2 v - (\beta+1)u,  \qquad 
	v_t = \gamma\Delta_\nu v + \beta u - u^2v. 	
\end{nalign}
By applying results from \cite{CMCKS01, CMCKS02}, we obtain these two \rev{systems} (considered in bounded domains) have stationary solutions of  both types: regular and discontinuous. 
However, their all nonconstant stationary solutions  are unstable.   

Next, in Section \ref{sec:stable}, we discuss two other models: the diffusion-degenerated Oregonator 
\begin{nalign}
	u_t = u - u^2 + \alpha v\frac{\beta-u}{\beta+u}, \qquad v_t = \gamma \Delta_\nu v + u - v
\end{nalign}
and a certain kind of predator-prey system
\begin{nalign}
	u_t = u \left(\frac{u^2}{u^3 + 1}  - v\right),\qquad
	v_t = \gamma\Delta_\nu v  + v (\alpha u -  v - \beta ).
\end{nalign}
In the following, we show how to use theory from the work \cite{CMCKS02} to construct {\it stable discontinuous stationary solutions} of these two systems.

\section{General results} 
\label{sec:Results}
Here, we review briefly results from \cite{CMCKS01,CMCKS02} formulating  them in a form which is suitable to deal with a system consisting of one ODE coupled with one reaction-diffusion equation.

We deal with 
a weak solution 
$( U, V) = \big(  U(x), V(x) \big)$
to the boundary value problem
\begin{nalign}
	\label{DisProbDef}
	f(U,V)&=0,  &&     &&x\in\overline{\Omega},  &&\\
	\gamma \Delta_\nu V+g(U,V)&=0,  &&  &&x\in \Omega,
\end{nalign}
with arbitrary $C^2$-functions $f$ and $g$, with constant $\gamma>0$, and \rev{in an} open bounded domain $\Omega\subseteq \R^N $ with a $C^2$-boundary. 
Recall ({\it cf.}~\cite{CMCKS01,CMCKS02}) that  a weak solution has the following properties: $U\in\Li$, $V\in \W12$, first equation in \eqref{DisProbDef} holds true almost everywhere and the second equation is satisfied in \rev{the} usual weak sense in the \rev{Sobolev space} $\W12$.  

We require that problem \eqref{DisProbDef} has a constant solution, namely,  a \rev{constant vector}
\begin{nalign}
	(\overline{  U}, \overline{ V}) \in \R^2 \quad\text{such that}\quad   f(\overline{U}, \overline{V}) = 0  \quad\text{and} \quad g(\overline{U}, \overline{V}) = 0.
\end{nalign} 
In this work, we use the following notation 
\begin{nalign}
	\label{ass:Matrix}
	 a_0 = f_{u}\left( \overline{U}, \overline{ V} \right),\quad  b_0 =  f_{v}\left( \overline{U}, \overline{ V} \right), \quad c_0 =g_{u}\left( \overline{U}, \overline{ V} \right), \quad d_0 = g_v\left( \overline{U}, \overline{ V} \right)
\end{nalign}	
and we always impose the following assumption. 

\begin{assumption}
	\label{ass:a0}
	We assume that $a_0=f_u(\oU,\oV) \neq 0$. 
\end{assumption}

\subsection{Regular stationary solutions}

 A solution $\UVx$ of system \eqref{DisProbDef} is called regular if there exists  a $C^2$-function $ k: \R\to \R$ such that $ U(x)= k(V(x))$ for all $x\in \Omega$ and
 \begin{nalign}
 	f\big( U(x),V(x)\big) =  f\big( k(V(x)),V(x)\big)=0 \quad \text{for all}\quad x\in \Omega,
 \end{nalign}
 where $V=V(x)$ is a solution of  the elliptic Neumann problem  
 \begin{align}
 	\gamma\Delta_\nu V+h(V)=0\qquad   \text{for}\quad x\in \Omega 
 \end{align}
 with $h(V)=g\big( k(V),V\big)$. \rev{Now,} we recall a result from work \cite{CMCKS01} on the existence of regular stationary solutions in the case of one ODE coupled with one reaction-diffusion equation.

\begin{prop}[{\cite[Prop.~2.5]{CMCKS02}}] 
	\label{thm:reg}
	Let $N \leqslant 6$ and let $(\overline{  U}, \overline{ V}) \in \R^2$ be a constant solution of problem~\eqref{DisProbDef} satisfying Assumption \ref{ass:a0}. If
	\begin{nalign}
		\label{reg:Det}
		\frac{1}{a_0 }\det\begin{pmatrix}
			a_0  & b_0  \\
			c_0  & d_0
		\end{pmatrix} =\gamma\mu_k> 0,
	\end{nalign}
	for some $\mu_k$ eigenvalues of $-\Delta_\nu$, then there exists a sequence of real numbers  $d_\ell\to 1$ such that  the following ``perturbed'' problem
	\begin{nalign}\label{seq1:d}
		f(U,V)=0,&&     \quad &x\in\overline{\Omega},  && \\
		d_\ell \Delta_\nu V+(1-d_\ell)(V-\overline{ V})+g( U,V)=0,&&  \quad &x\in \Omega
	\end{nalign}
	has a non-constant regular solution. 
\end{prop}

All regular stationary solutions are \rev{unstable,} except the degenerate case when $f_{ u}\UVx \leqslant 0$ for all $x\in\oo$ and $f_u\big(U(x_0), V(x_0)\big) = 0$ for some $x_0\in \oo$, see the following theorem.  
 
\begin{theorem} [{\cite[Thm.~2.7 and Thm.~2.8]{CMCKS01}}]
	\label{reg:Instability}
	Let $( U, V) = \big( U(x), V(x) \big)$ be a non-constant regular stationary solution of problem \eqref{eq1}-\eqref{Neumann} such that
	\begin{itemize}
		\item either $f_{ u}\big( U(x_0), \, V(x_0)\big) > 0$ for some  $x_0\in \oo$, \\
		\item or $f_{ u}\UVx < 0$ for all  $x\in \oo$  and the domain $\Omega$ is convex.
	\end{itemize}
	Then $( U, V)$ is nonlinearly unstable. 
\end{theorem}

\subsection{Discontinuous solutions}

\rev{Now, we impose additional assumptions} on a constant stationary solution  using notation \eqref{ass:Matrix}. 

\begin{assumption}
	\label{ass:StatSol}
	The constant solution  $(\oU, \oV) \in \R^2$ of problem \eqref{DisProbDef} satisfies
	\begin{nalign}
		\label{Sineq}
		\frac{1}{ a_0 }\det\begin{pmatrix}
			a_0 &  b_0 \\
			c_0  &  d_0
		\end{pmatrix} \neq \gamma\mu_k,
	\end{nalign}
	for each eigenvalue $\mu_k$ of $-\Delta_\nu$.	
\end{assumption}

Notice that if $ a_0 = f_{u}\left( \overline{U}, \overline{ V} \right) \neq 0$, by the Implicit mapping theorem, there exist an open neighborhood $\V \subseteq \R$ of $\overline{ V}$ and a function $k \in C^2(\V,\R)$ such that
\begin{nalign}
	k(\overline{V}) = \overline{U} \quad \text{and} \quad f(k(w),w) = 0\quad \text{for all } w\in \V.
\end{nalign}

In the following assumption,  the equation $f(U, V) = 0$ is required to have  \textit{two different branches of solutions} with respect to $U$ on a common domain $\V$.

\begin{assumption}
	\label{ass:DiscontinuousStationaryTwoBranches}
	We assume that there exists an open set $\V \subseteq \R$ and $k_1, k_2 \in C^2(\V, \R)$ such that
	\begin{itemize}
		\item $ \overline{ V} \in \V$,
		\item $\overline{ U} = k_1 (\overline{ V}), \quad k_1 (\overline{ V} ) \neq k_2(\overline{ V})$, 
		\item $ f\big(k_1(w),w)\big) = f\big( k_2(w),w)\big) = 0 \ \text{for all} \ w\in \V.$
	\end{itemize}	
\end{assumption}

Such two branches $k_1, k_2$ allow us to construct discontinuous solution of problem~\eqref{DisProbDef}.

\begin{theorem}[Existence of discontinuous stationary solutions {\cite[Thm.~2.6]{CMCKS02}}]
	\label{DisExBan}
	Assume that
	\begin{itemize}
		\item problem \eqref{DisProbDef}  has a constant solution $(\overline{   U}, \overline{ V}) \in \R^2$ satisfying \rev{Assumptions} \ref{ass:a0} and \ref{ass:StatSol},
		\item \rev{Assumption} \ref{ass:DiscontinuousStationaryTwoBranches} holds true,
		\item \rev{for an arbitrary open set $\Omega_1 \subset \Omega$ we put $\Omega_2 = \Omega\setminus\overline{\Omega}_1$} \rev{and the  sets $\Omega_1$ and $\Omega_2$ have  nonzero Lesbegue measures.}
	\end{itemize}
	There is $\varepsilon_0 >0$ such that for all $\varepsilon \in (0, \,\varepsilon_0)$ there exists $\delta >0$ such that if $|\Omega_2|< \delta$ 
	then problem \eqref{DisProbDef} has a weak solution  with the following properties:
	\begin{itemize}
		\item $( U, V)\in L^\infty(\Omega) \times \W2p$ for each $p\in[2,\infty)$,
		\item $V=V(x)$ is a weak solution to problem
		\begin{nalign}
			\label{eq:DinsontinousSolutionsVEquation}
			\gamma\Delta_\nu V + g( U,V) = 0 \quad \text{for} \quad x\in \Omega,
		\end{nalign}
		where
		\begin{nalign}
			\label{Uform}
			 U (x) = \begin{cases}
				 k_1\big(V(x)\big), \quad x\in \Omega_1, \\
				 k_2\big(V(x)\big), \quad x\in \Omega_2,
			\end{cases}
		\end{nalign} 
		satisfies $ f\big( U(x),V(x)\big) = 0$ for almost all $x\in \overline{\Omega}$.
		\item the solution $\UV$ stays around the points $(\oU, \oV)$ and \rev{$\big(k_2(\oV),\oV\big)$} in the following sense
		\begin{nalign}
			\label{DisVar}
			\|V - \overline{ V} \|_\Li < \varepsilon \quad \text{and} \quad \| U -   \oU \|_{L^\infty(\Omega_1)} + \| U -   k_2 (\overline{ V}) \|_{L^\infty(\Omega_2)}  < C\varepsilon
		\end{nalign}
		for a constant $C = C( f,  g)$. 
	\end{itemize}
\end{theorem}
Next, we discuss a stability of discontinuous stationary solutions by the linearization procedure. Here, we say that $\UV$ is {\it linearly stable (unstable)}  if the zero solution to the following system
\begin{nalign}
	\label{eq:lin}
	\frac{\partial}{\partial t}\begin{pmatrix}  \varphi \\ \psi \end{pmatrix} &= \begin{pmatrix} 0 \\ \gamma\Delta_\nu \psi \end{pmatrix} + \begin{pmatrix}  a &  b \\  c & d \end{pmatrix} \begin{pmatrix}  \varphi \\ \psi \end{pmatrix}   
	\equiv \L_p\begin{pmatrix}  \varphi \\ \psi \end{pmatrix}  
\end{nalign}
with the bounded  (and possibly $x$-dependent) coefficients 
\begin{nalign}
	\label{eq:LinSys}
	a = f_{u}\left( {U}, { V} \right),\quad  b =  f_{v}\left( {U}, { V} \right), \quad c =g_{u}\left( {U}, { V} \right), \quad d = g_v\left( {U}, { V} \right),
\end{nalign}
  is stable (unstable).

\begin{theorem}[Instability of stationary solutions \cite{CMCKS01,MR3600397}] 
	\label{thm:InstabilityAutocata}
	Let $\UV$ be an arbitrary weak stationary solution to problem \eqref{eq1}-\eqref{Neumann}. If
	\begin{nalign}\label{autocat}
		\rev{f_{ u} \big( U(x), \, V(x) \big)> 0 } \qquad \text{for $x$ from a set of a positive measure} 
	\end{nalign}
	then $\UV$ is linearly unstable in $\Lp{}^2$ for each $p\in(1,\infty)$. 
\end{theorem}

 The assumption \eqref{autocat} is called the {\it autocatalysis condition} in the work \cite{MR3600397}. It appears in models with the Turing instability and  it leads to an instability of the both: constant and nonconstant stationary solutions. Theorem  \ref{thm:InstabilityAutocata} is an immediate consequence of the results in \cite[Thm.~4.5]{MR3600397} and 
 \cite[Thm.~4.6]{CMCKS01}, where the spectrum of the operator 
$\L_p$ defined in \eqref{eq:lin} and acting on the space $L^p(\Omega)^2$ is described.

Next, we discuss a stability of discontinuous stationary solutions and we require
 from the constant solution $(\overline{  U}, \overline{V})\in \R^2$ used in Theorem \ref{DisExBan} to have the following \rev{additional properties}.  

\begin{assumption}
	\label{ass:LinearStability}
	\rev{For the numbers defined in \eqref{ass:Matrix},} assume that the following inequalities \rev{hold} true
	\begin{nalign}
		 a_0<0, \quad d_0<0 \quad \text{and the matrix } \quad \begin{pmatrix}  a_0 &  b_0 \\  c_0 & d_0, \end{pmatrix} 
	\end{nalign} 
	 has both  eigenvalues with strictly negative real parts. 
\end{assumption}

\begin{theorem}[Stability of discontinuous stationary solutions \cite{CMCKS02}]
	\label{thm:ApplicationSystemStabilityDiscontinuous2}
	Let the assumptions of Theorem \ref{DisExBan} \rev{hold} true.  If, moreover,
	\begin{itemize}
		\item Assumption \ref{ass:LinearStability} is valid,
		\item $\varepsilon>0$ and $|\Omega_2|>0$ are small enough in Theorem \ref{DisExBan},
		\item the following inequalities are satisfied
		\begin{nalign}
			\label{eq:SecBranch}
			  f_{ u}  ( k_2(\overline{V}), \, \overline{ V}  \big) < 0 \quad \text{ and } \quad 	 g_v\big(  k_2 (\overline{V}), \overline{V}\big) < 0.
		\end{nalign}
	\end{itemize}
	then the stationary solution $\UV$ constructed in Theorem \ref{DisExBan} is exponentially stable in $\Li$.
\end{theorem}

Theorem \ref{thm:ApplicationSystemStabilityDiscontinuous2} can be obtained from 
results in \cite{CMCKS02} in the following way. First, we apply~\cite[Prop.~4.2]{CMCKS02} to linear system \eqref{eq:lin} to show that the stationary solution is linearly exponentially stable in $\Lp{}$. Here, our system consists of one ODE coupled with one PDE, thus Assumption \ref{ass:LinearStability} and the condition $d_0<0$ imply that \cite[Assumptions~2.8 and 2.9]{CMCKS02}  are satisfied by \cite[Rem. 2.16]{CMCKS02}. Next, the nonlinear stability in~$\Li$ is a direct consequence of \cite[Thm. 2.13]{CMCKS02} which proof does not require \rev{from} $\gamma>0$ to be large enough. 

\begin{rem}
	\label{thm:3Branches}
	Other discontinuous stationary solutions can be constructed under \rev{the} following more general version of Assumption \ref{ass:DiscontinuousStationaryTwoBranches}, where we postulate an existence of a set $\V\subseteq \R$ and different branches $ k_1, \cdots,  k_J \in C(\V, \R)$ of solutions to \rev{the} equation $ f(  U, V) = 0$.  Then for \rev{an} arbitrary decomposition 
	\begin{nalign}
		\rev{\Omega \subseteq \overline{\bigcup_{i\in \lbrace 1,\cdots, J \rbrace} \Omega_i}}, \quad \Omega_i \cap \Omega_j =\emptyset \ \ \text{for} \ \ i\neq j \quad \text{and} \quad |\Omega_i| <\delta \ \ \text{for} \ \ i\in\lbrace 2, \cdots, J\rbrace,  
	\end{nalign} 
	we can construct a discontinuous stationary solution of the form
	\begin{nalign}
		\label{Uformn}
		 U (x) = \begin{cases}
			 k_1\big(V(x)\big), \quad x\in \Omega_1, \\
			\qquad \vdots \\
			 k_J\big(V(x)\big), \quad x\in \Omega_J,
		\end{cases}
	\end{nalign} 
	using the same reasoning as in Theorem \ref{DisExBan}. If 
	 \begin{nalign} 
	 	f_{ u}  \big( k_i(\overline{V}), \, \overline{ V} \big)  < 0\quad \text{and} \quad g_v\big(  k_i (\overline{V}), \overline{V}\big) < 0\quad  \text{for each} \quad i\in\lbrace 1, \cdots, J\rbrace
	 \end{nalign} 
	 then this solution is linearly and nonlinearly stable.

\end{rem}
\section{Unstable stationary solutions} 
\label{sec:Instability}

We apply results from Section \ref{sec:Results} 
to particular reaction-diffusion-ODE systems. The following two models with classical nonlinearities have stationary solutions $\UV$, the both regular and discontinuous. For each stationary solution the
autocatalysis condition~\eqref{autocat} is satisfied, \rev{hence such solutions are} unstable.

\subsection{Gray-Scott type model}
First, we deal with the reaction-diffusion-ODE model with the Gray-Scott nonlinearities
\begin{nalign}
	\label{GrayScottODE}
	u_t &=  u^2 v - \alpha u, & x\in \overline{\Omega}, && t>0,\\
	v_t & = \rev{\gamma \Delta_\nu v - u ^2 v + \beta(1-v) }, &\quad x\in {\Omega}, && t>0,
\end{nalign}
with arbitrary constants $\alpha>0$, $\beta>0$ and the diffusion coefficient $\gamma>0$. This is the diffusion degenerate case of the classical Gray-Scott model introduced in  \cite{GRAY19841087} and studied in \textit{e.g.}~\cite{MR2256841, MR1878337} and in the references therein. 

Notice that problem \eqref{GrayScottODE} has the constant \rev{stationary} solution $(\oU_1,\oV_1) = (0,1)$. Two \rev{other} constant solutions $(\oU_2,\oV_2)$, $(\oU_3,\oV_3)$ satisfy the relations
\begin{nalign}
	\label{GSQuadratic}
	\oU &= \frac{\alpha}{\oV} \quad \text{and} \quad 
	0 &=  \oV^2 - \oV + \frac{\alpha^2}{\beta}
\end{nalign}
and they exist if $\alpha^2/\beta < 1/4$.

\begin{theorem}
	\label{GrayScottRegular}
	For a certain choice of coefficients $\alpha, \beta>0$ and for a discrete sequence of diffusion coefficients $\gamma>0$ problem \eqref{GrayScottODE} has a regular stationary solution. All regular stationary solutions {\rm (}not only those obtained via Proposition \ref{thm:reg}{\rm )} are nonlinearly unstable.
\end{theorem}

\begin{proof}
	We choose $\alpha,\beta >0$ such that $\alpha^2/\beta < 1/4$ and consider a solution $\oV_3$ of the quadratic equation in \eqref{GSQuadratic} satisfying $\overline{ V}_3<1/2$.
	We apply Proposition \ref{thm:reg} with the constant stationary solution $(\oU,\oV)=(\oU_3,\oV_3)$, where $\oU_3 = \alpha /\oV_3$. Assumption~\ref{ass:a0} holds true, because $a_0 = f_u (\oU_3, \oV_3) = \alpha>0$.	Moreover, for the numbers $a_0, b_0, c_0, d_0$ as in \eqref{ass:Matrix}, we have 
	\begin{nalign}\rev{
		\det\begin{pmatrix}
			a_0 &b_0 \\ c_0 & d_0
		\end{pmatrix} =  \begin{vmatrix}
			\alpha & (\alpha/\overline{V}_3)^2 \\ - 2\alpha & - (\alpha/\overline{  V}_3)^2 - \beta 
		\end{vmatrix} =\alpha\left((\alpha/\overline{V}_3)^2 - {\beta}\right) > 0},
	\end{nalign}   
	where the last inequality follows \rev{from} an explicit formula for $\overline{ V}_3$. Thus, we may choose $\gamma>0$ to satisfy equation \eqref{reg:Det} for some eigenvalue $\mu_k>0$.
	
	Next, we show that every stationary solution satisfies $f_u\big(U(x),V(x)\big) \neq0$ for each $x\in \Omega$. Indeed, notice that either $U(x) = 0$ for all~$x\in \oo$ or $U(x) = \alpha /V(x)$ for all~$x\in \oo$. Thus, since, $f_u(U,V) = 2UV - \alpha$ we have either $f_u\big(U(x),V(x)\big) = \alpha$ or $f_u\big(U(x),V(x)\big) = -\alpha$. 
	
	\rev{Hence,} by \rev{Theorem \ref{reg:Instability},} we conclude that all regular solutions to problem \eqref{GrayScottODE} are unstable. 
\end{proof}

\begin{theorem}
	For arbitrary $\alpha>0$ and $\beta>0$ and for each diffusion \rev{coefficient $\gamma>0$,} except of a discrete set, there exists a family of discontinuous stationary solutions to problem \eqref{GrayScottODE}. All discontinuous stationary solutions {\rm (}not only those obtained via Theorem \ref{DisExBan}{\rm )} are linearly unstable in $\Lp{}^2$ for each $p\in (1,\infty)$.
\end{theorem}

\begin{proof}
	First, we apply Theorem \ref{DisExBan} to construct discontinuous stationary solutions to  problem \eqref{GrayScottODE}. Assumption \ref{ass:StatSol} is satisfied for \rev{all stationary points} $(\overline{   U}_1, \overline{ V}_1)$, $(\overline{   U}_2, \overline{ V}_2)$, $(\overline{   U}_3, \overline{ V}_3)$ \rev{and for all diffusion coefficients $\gamma>0$, possibly except of  a discrete set}. To check Assumption \ref{ass:DiscontinuousStationaryTwoBranches}, we notice that the equation \rev{${U}(U V - \alpha)=0$} has \rev{the following} two branches of solutions
	\begin{nalign}
		U = k_1(V) =  0 \qquad \text{and} \qquad U = k_2(V) = \frac{\alpha}{V} \quad \text{for all} \quad V\neq 0.
	\end{nalign}
	\rev{Thus, for an arbitrary open set $\Omega_1 \subset \Omega$, with $\Omega_2=\Omega\setminus\overline{\Omega}_1$} and such that the measure $|\Omega_2|>0$ is sufficiently small,  there exist a family of discontinuous stationary solutions around 
	the points  either $(\oU_1, \oV_1)$ or $(\oU_2, \oV_2)$ or $(\oU_3, \oV_3)$ as stated in Theorem~\ref{DisExBan}. 
	
	Each discontinuous stationary solution $\UV$ satisfies $k_2\big(V(x)\big) = \alpha/V(x)$ for $x$ from a set of a positive measure. Thus, $f_u\big(k_2(V(x)),V(x)\big) = \alpha>0$ and this solution is linearly unstable in $\Lp{}^2$ for each $p\in(1,\infty)$  by Theorem~\ref{thm:InstabilityAutocata}.
\end{proof}

\subsection{Brusselator type model}
Next, we deal with the reaction-diffusion-ODE model with the nonlinearity  as in the Brusselator system
\begin{nalign}
	\label{BrusselatorODE}
	u_t &=  \alpha + u^2 v - (\beta+1)u, & x \in \oo, && t>0, \\
	v_t &= \gamma\Delta_\nu v + \beta u - u^2v, & x\in \Omega, && t>0, 	
\end{nalign}
with arbitrary constants $\alpha>0$, $\beta>0$ and the diffusion coefficient $\gamma>0$. This is the diffusion-degenerate case of the classical Brusselator model 
which was introduced in~\cite{PL68}  as a model for
an autocatalytic oscillating chemical reaction. 
 It has been
shown ({\it e.g.} \rev{in \cite{MR722938,MR2200795,MR1869203} and in} references therein) that the Brusselator system with the both nonzero diffusion coefficients (as in \eqref{RD:gen})
exhibits the Turing instability.

Notice that a constant stationary solution  $(\overline{U}, \overline{  V}) \in \R^2$ of problem \eqref{BrusselatorODE} satisfies the equations
\begin{nalign}
	\label{BrusselatorStat}
	0 = \alpha + \oU^2\oV - (\beta+1)\oU, \qquad 0= \beta\oU - \oU^2\oV
\end{nalign}
with the only solution 
\begin{nalign}
	(\oU_1, \oV_1) = (\alpha, \beta/\alpha).
\end{nalign} 

\begin{theorem}
	For each $\alpha>0$ and $\beta>1$ and for a discrete sequence of \rev{the} diffusion coefficients $\gamma>0$ problem \eqref{BrusselatorODE} has a regular stationary solution. All regular stationary solutions {\rm (}not only those obtained via Proposition \ref{thm:reg}{\rm )} are nonlinearly unstable.
\end{theorem}

\begin{proof}
	We use the notation from \eqref{ass:Matrix} with the constant stationary solution $(\oU_1,\oV_1)$ and we notice that 
	\begin{nalign}
		a_0 = \beta - 1 \quad \text{and} \quad \det\begin{pmatrix}
			a_0 & b_0 \\ c_0 & d_0
		\end{pmatrix} = \alpha^2.
	\end{nalign}
	Thus, for \rev{$\beta>1$,} the proof of \rev{the existence of regular stationary solutions,} may be completed in the analogous way as the proof of Theorem \ref{GrayScottRegular}. \rev{All regular stationary solutions are unstable by Theorem \ref{reg:Instability}.}
\end{proof}

\begin{theorem}
	For \rev{some} $\alpha>0$, $\beta>0$, $\beta\neq 1$, and for each diffusion coefficient $\gamma>0$ except of a discrete set, there exists a family of discontinuous stationary solutions to problem \eqref{BrusselatorODE}. All discontinuous stationary solutions {\rm (}not only those obtained via Theorem \ref{DisExBan}{\rm )} are linearly unstable in $\Lp{}^2$ for each $p\in (1,\infty)$.
\end{theorem}

\begin{proof}
	We choose \rev{the} coefficient $\alpha>0$, $\beta>0$ and $\gamma>0$ \rev{in such a way that} Assumptions \ref{ass:a0} and~\ref{ass:StatSol} are satisfied (notice that $a_0  = \beta - 1 \neq 0$). The equation 
	\begin{nalign}
		f(U,V) = \alpha + U^2V - (\beta+1)U = 0
	\end{nalign}
	has two different branches of solutions given by the explicit formulas
	\begin{nalign}
		U= k_1(V) = \frac{\beta + 1 + \sqrt{(\beta+1)^2 - 4\alpha V}}{2V}, \quad U= k_2(V) = \frac{\beta + 1 - \sqrt{(\beta+1)^2 - 4\alpha V}}{2V}
	\end{nalign}
	for $V < (\beta+1)^2/(4\alpha)$. Note that 
	\begin{nalign}
		\oV_1 = \beta/\alpha < (\beta+1)^2/(4\alpha) \quad \text{and} \quad  \oU_1 = k_1(\oV_1).
	\end{nalign}
	 \rev{Thus, for an  open set $\Omega_1 \subset \Omega$ with $\Omega_2= \Omega \setminus \overline{\Omega}_2$ and such that the measure $|\Omega_2|>0$ is} sufficiently small,  there exist a family of discontinuous stationary solutions around \rev{points} $(\oU_1, \oV_1)$ and $(k_2(\oV_1), \oV_1)$ as stated in Theorem~\ref{DisExBan}.
	
	Each discontinuous stationary solution $\UV$ satisfies $U(x) =  k_1\big(V(x)\big)$ for $x$ from a set of positive measure. Thus, 
	\begin{nalign}
		f_u\big(k_1(V(x)),V(x)\big) = \sqrt{(\beta+1)^2 - 4\alpha V(x)}>0
	\end{nalign} 
	and this solution is linearly unstable in $\Lp{}^2$ for each $p\in(1,\infty)$  by Theorem~\ref{thm:InstabilityAutocata}.
\end{proof}

\section{Stable discontinuous stationary solutions}
\label{sec:stable}

Now, we discuss reaction-diffusion-ODE models with stable discontinuous stationary solutions. 

\subsection{Oregonator type model} 
The following reaction-diffusion-ODE model 
\begin{nalign}
	\label{OregonatorODE}
	u_t &= u - u^2 + \alpha v\frac{\beta-u}{\beta+u}, & \rev{x\in\oo,}  && t>0,\\
	v_t &= \gamma \Delta_\nu v + u - v, & x\in \Omega, && t>0,
\end{nalign}
contains the nonlinearities as in the classical Oregonator system with the parameters $\alpha>0$, $\beta>0$ and the diffusion coefficient $\gamma>0$. A general version of this model was introduced in \cite{doi:10.1063/1.440418,doi:10.1063/1.1681288,Field1972OscillationsIC} to describe the FKN-Belousov reaction. It is known that system \eqref{OregonatorODE} may exhibit the Turing instability for a certain choice of parameters and diffusion coefficients, see \textit{e.g.}~\cite{ZHOU2016192, PENG20102337} and the references therein.

System \eqref{OregonatorODE} has the zero constant solution $(\oU_1, \oV_1) = (0,0)$ and two other constant solutions have the form 
\begin{nalign}
	(\oU_2, \oV_2) = (\oU_2,\oU_2) \qquad (\oU_3, \oV_3) = (\oU_3, \oU_3),
\end{nalign}
where \rev{the numbers} $\oU_2$ and $\oU_3$ satisfy the quadratic equation
\begin{nalign}
	\oU^2 +\oU(\beta+\alpha-1)-\beta(\alpha+1) = 0.
\end{nalign}
Fig.~\ref{fig:Orego} presents a location of the  nullclines defined by the following equation for $\beta>1$ and for sufficiently small $\alpha>0$: 
\begin{nalign}
	\label{eq:OregNullclines}
	f(U,V) = U - U^2 + \alpha V \frac{\beta -U}{\beta+U} = 0 \quad \text{and} \quad g(U,V) = U-V = 0.
\end{nalign}

\begin{figure}[h]
	\includegraphics[width=0.7\linewidth]{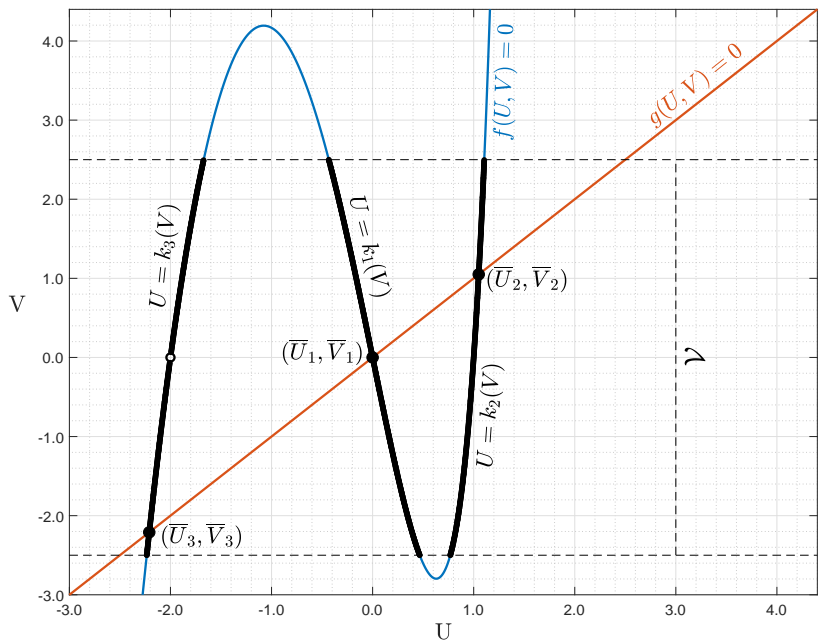}	
	\caption{The nullclines \eqref{eq:OregNullclines} for the Oregonator model with $\beta=2$ and sufficiently small $\alpha>0$. The point $(-2,0)$ is excluded because it does not satisfy the equation $f(U,V) = 0$.}
	\label{fig:Orego}
\end{figure}

\begin{theorem}
	\label{prop:OregEx}
	\rev{For each $\beta>1$ and sufficiently small $\alpha>0$} there exist a family of discontinuous stationary solutions to problem \eqref{OregonatorODE}.
\end{theorem}

\begin{proof}
	\rev{Fix $\beta>1$ and choose sufficiently small $\alpha>0$ to have the nullclines as in Fig.~\ref{fig:Orego}. In such a case, we have an open subset $\V\subset \R$ satisfying $V_1, V_2, V_3 \in \V$ and three branches $k_1, k_2, k_3$ defined on $\V$}.
	
	Here, we check assumptions of Theorem \ref{DisExBan} for the stationary points either $(\oU_1, \oV_1)$ or $(\oU_2, \oV_2)$ or $(\oU_3, \oV_3)$ in the cases shown in Fig. \ref{fig:Orego}. \rev{First,} we notice that for \rev{a branch} $k\in \lbrace k_1, k_2, k_3\rbrace$, we have
	\begin{nalign}
		\label{eq:DerIden}
		\frac{d}{dV}f\big(k(V), V\big) = f_u\big(k(V), V\big)k'(V) + f_v\big(k(V), V\big) = 0.
	\end{nalign}
	Since $$f_v(U,V) = \frac{\alpha(\beta - U)}{\beta + U},$$ choosing \rev{the} constant stationary solution \rev{$(\oU_i,\oV_i)\in\R^2$} satisfying $\oU\neq \beta$ and $k'(\oV) \neq 0$ we obtain $a_0 = f_u(\oU,\oV) \neq 0 $ as required in Assumption \ref{ass:a0}. This is the case of all constant stationary solutions $(\oU_1, \oV_1)$, $(\oU_2, \oV_2)$ and $(\oU_3, \oV_3)$ shown in Fig. \ref{fig:Orego}.
	
	To apply Theorem \ref{DisExBan}, we need a constant stationary solution $(\oU_i,\oV_i)$, an open set $\V\subset \R$ with $V_i \in \V$ and two branches of solutions $k_i,k_j\in C(\V,\R)$ to equation $f(U,V) = 0$. On Fig \ref{fig:Orego}, we may choose each constant stationary solution with corresponding branches, namely, \rev{the point} $(\oU_1, \oV_1)$ with the branches $k_1$ and  $k_2$, the point $(\oU_2, \oV_2)$ with the branches $k_2$ and either $k_1$ or $k_3$ as well as \rev{the point} $(\oU_3, \oV_3)$ with the branches $k_3$ and either $k_1$ or $k_2$. We do not choose the point $(\oU_1, \oV_1)$ with the branch $k_3$ in order to avoid the singular point $(-\beta, 0)$.
	
	Thus, for $\gamma>0$ satisfying Assumption \ref{ass:StatSol}, \rev{an arbitrary open set $\Omega_1 \subset \Omega$ with $\Omega_2 = \Omega\setminus \overline{\Omega}_1$} and $|\Omega_2|>0$ sufficiently small,  there exist a family of discontinuous stationary solutions around \rev{the points} $(\oU, \oV) = (k_i(\oV), \oV)$ and $(k_j(\oV), \oV)$ as stated in Theorem~\ref{DisExBan}. 
	
	By Remark \ref{thm:3Branches}, we \rev{can also construct} discontinuous stationary solutions with all three branches $k_1$, $k_2$, $ k_3$ and \rev{with an} arbitrary partition $\Omega_1$, $\Omega_2$, $\Omega_3$ of $\Omega$. 
\end{proof}

\begin{theorem}
	\label{thm:OregonatorStable}
	\rev{For each $\beta>1$ and {for} sufficiently small $\alpha>0$} there exist  discontinuous stationary solutions constructed around the points $(\oU_2,\oV_2)$ and $\big(k_3(\oV_2),\oV_2\big)$ which are stable. 
\end{theorem}

\begin{proof}	
	\rev{As in the proof of Theorem \ref{prop:OregEx},  we have an open subset $\V\subset \R$ satisfying $V_1, V_2, V_3 \in \V$ and three branches $k_1, k_2, k_3$ defined on $\V$, see  Fig. \ref{fig:Orego}}. Let $\UV$ be a discontinuous stationary solution constructed around points $(\oU_2,\oV_2)$ and $(k_3(\oV_2),\oV_2)$ via Theorem \ref{DisExBan}.
	
	First, we check Assumption~\ref{ass:LinearStability}. Since $|\oU_2|<\beta$ and \rev{since} $k_2(V)$ is increasing in the neighbourhood of $\oV_2$ (see Fig. \ref{OregonatorODE}) we have $a_0<0$ by \rev{equation}~\eqref{eq:DerIden}. In order to show that the matrix $\begin{spm}a_0 &b_0 \\ c_0 &d_0\end{spm}$ has both eigenvalues with negative real parts, it suffices to show that
	\begin{nalign}
		a_0 + d_0 < 0 \quad \text{and} \quad a_0d_0 - b_0c_0 >0. 
	\end{nalign}
	Since $c_0 = 1$ and $d_0 = -1$, first inequality is immediately satisfied. Moreover, since \rev{$k_2'\big(\oV_2)\big)<1$}, by identity \eqref{eq:DerIden}, we have
	\begin{nalign}
		\rev{a_0d_0 - b_0c_0 = -a_0\big(1-k_2'(\oV_2)\big)>0}.
	\end{nalign}
	
	Finally, $|k_3\big(\oV_2\big)|<\beta$ (see Fig. \ref{OregonatorODE}) and the function $k_3(V)$ is increasing in \rev{a} neighbourhood of $\oV_2$ which implies that $f_u\big(k_3(\oV_2),\oV_2\big)<0$.
	
	By Theorem \ref{thm:ApplicationSystemStabilityDiscontinuous2}, the discontinuous stationary solution $\UV$ is stable provided $\varepsilon>0$ and $|\Omega_2|>0$ are small enough.
\end{proof}

\begin{rem}
	On the other hand, all discontinuous stationary solutions constructed in Theorem~\ref{OregonatorODE} around the points $(\oU_1,\oV_1)$ and $(\oU_3,\oV_3)$ are unstable. Indeed, analogously as in the proof of Theorem \ref{prop:OregEx}, we analyse nullclines \eqref{eq:OregNullclines} presented in Fig. \ref{fig:Orego}. Following the calculations from equation \eqref{eq:DerIden} we obtain, that \rev{the number} $f_u\big(k(\oV),\oV)$ is positive if either 
	\begin{nalign}
		\label{eq:OregA0}
		|\oU|<\beta \quad \text{and} \quad k'(\oV)<0 \qquad \text{or} \qquad |\oU|>\beta \quad \text{and} \quad k'(\oV)>0.
	\end{nalign}
	Since the constant solutions given by explicit formulas satisfy $\oU_1=0$,  $\oU_3<-\beta$ and \rev{since} the function $k(V)$ is decreasing in \rev{a} neighbourhood of $\oU_1$ and increasing in a neighbourhood of $\oU_3$ (see Fig. \ref{fig:Orego}) we obtain that every solution constructed around points $(\oU_1, \oV_1)$ and $(\oU_3, \oV_3)$ is linearly unstable by Theorem \ref{thm:InstabilityAutocata}. 
\end{rem}

\subsection{Predator-prey model}
Finally, we discuss stationary solutions to the following model of predator-prey interactions which is a particular case of a system considered in \cite{MR579554}:
\begin{nalign}
	\label{PredatorPreyODE}
	u_t &= u \left(\frac{u^2}{u^3 + 1}  - v\right), & x\in\overline{\Omega}, && t>0 \\
	v_t &= \gamma \Delta_\nu v  + v (\alpha u -  v - \beta ) &  x\in \Omega, && t>0.
\end{nalign}
which $\alpha>0$ and $\beta>0$. Here, the couple $(\oU_2,\oV_2) = (0,0)$ is the constant stationary solution and the second constant stationary solution $(\oU_1,\oV_1)$ satisfies the system 
\begin{nalign}
	\label{eq:PredatorPreyNullclines}
	f(U,V) = \frac{U^2}{U^3 + 1}  - V = 0 \quad \text{and} \quad g(U,V) =  \alpha U -  V - \beta = 0. 
\end{nalign} 
\begin{figure}[h]
	\includegraphics[width=0.7\linewidth]{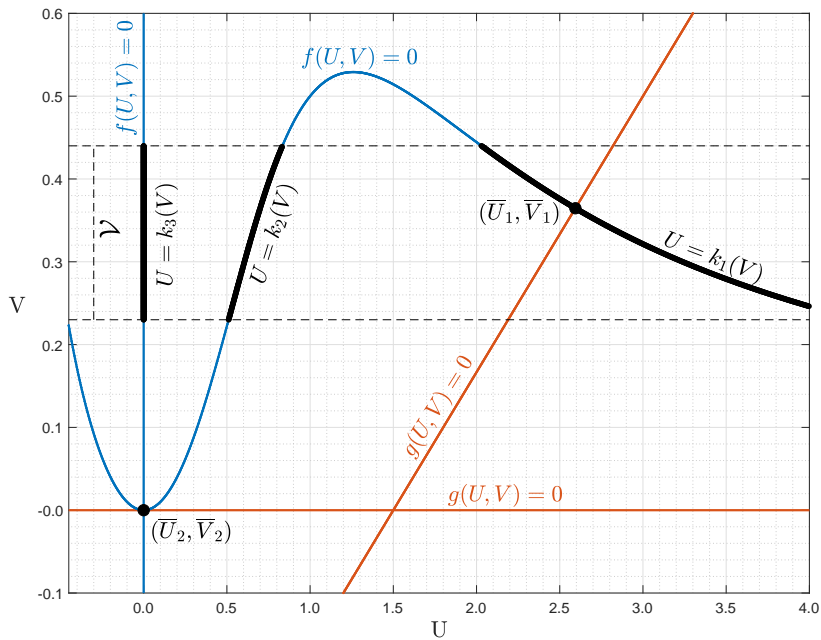}
	\caption{The nullclines \eqref{eq:PredatorPreyNullclines} for the predator-prey model with $\oU_1>U_m$.}
	\label{fig:PredatorPreyNullclines}
\end{figure}

\noindent
The nullclines defined by equations \eqref{eq:PredatorPreyNullclines} are sketched in Fig. \ref{fig:PredatorPreyNullclines}. Notice that the function ${U^2}/({U^3 + 1})$ attains a maximum for $U>0$ at some point which we denoted by $U_m$.

\begin{theorem}
	\label{thm:PredatorPrey}
	Assume that $\alpha >0$ and $\beta>0$ are such that the constant stationary solution \rev{$(\oU_1,\oV_1)$} satisfies \rev{$\oU_1\neq U_m$}. Then there exist a family of discontinuous stationary solutions to problem \eqref{PredatorPreyODE}. 
\end{theorem} 

\begin{proof}
	It is sufficient to apply Theorem \ref{DisExBan} in the same way as \rev{in the proof of} Theorem~\ref{prop:OregEx} with the stationary point \rev{$(\oU_1,\oV_1)$} with the branch \rev{$k_1$} and the branches either \rev{$k_2$} or $k_3$ (or the both, see Remark \ref{thm:3Branches}). 
\end{proof}

\begin{theorem}
	Assume that $\alpha >0$ and $\beta>0$ are such that the constant stationary solution \rev{$(\oU_1,\oV_1)$} satisfies \rev{$\oU_1 > U_m$}. Then there exist stable discontinuous stationary solutions to problem \eqref{PredatorPreyODE}. 
\end{theorem}

\begin{proof}
	Following the reasoning \rev{from} Theorem \ref{thm:OregonatorStable} we obtain that the \rev{number $f_u(\oU_1,\oV_1)$} is negative because the function \rev{$k_1(V)$} is decreasing in a neighbourhood of \rev{$\oV_1$} (see Fig.~\ref{fig:PredatorPreyNullclines}). Moreover, since $b_0 <0$, $c_0>0$ and $d_0 <0$ (with the notation from \eqref{ass:Matrix}) we obtain that 
	\begin{nalign}
		a_0 + d_0 <0 \quad \text{and} \quad a_0d_0 - b_0c_d>0
	\end{nalign} 
	and Assumption \ref{ass:LinearStability} holds true. By a direct calculation, we obtain that 
	\begin{nalign}
		\rev{f_u(0,\oV_1)  = -\oV_1 <0 \quad \text{and}\quad g_v(0,\oV_1) = -2\oV_1 - \beta <0}.
	\end{nalign}
	\rev{Thus,} the solution is stable by Theorem \ref{thm:ApplicationSystemStabilityDiscontinuous2}.
\end{proof}

\subsection{Numerical illustrations of discontinuous stationary solutions}

\begin{figure}[b]
	\includegraphics[width=1.0\linewidth]{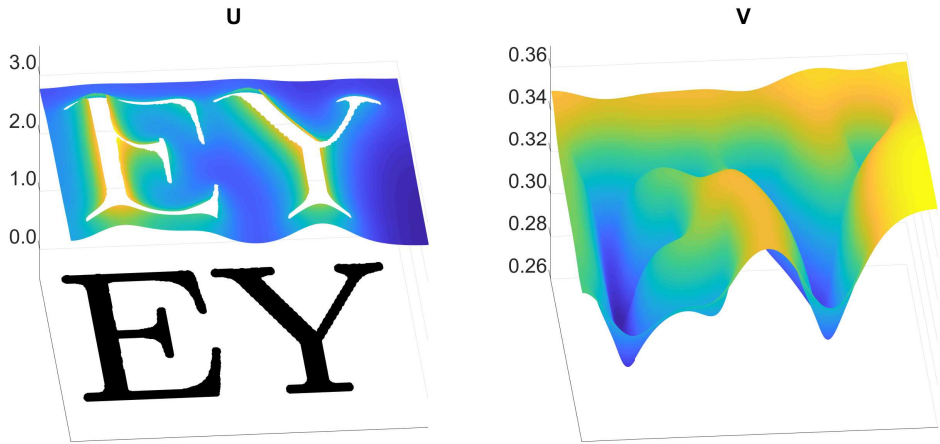} 
	\caption{The solution $u$, $v$ to problem \eqref{PredatorPreyODE} for sufficiently large $t>0$ and with the $EY$-shape initial condition satisfying equations~\eqref{ExampleIni}. The solution $u$ is equal to $0$ on the $EY$-shape set $\Omega_2$ (black) and stays close to $\oU_1$ on a set $\Omega_1$ (blue-green-yellow). The solution $v$ is an $EY$-perturbation of the constant solution $\oV_2$.}
	\label{fig:EY}
\end{figure}

We conclude this work by numerical simulations of solutions to problem \eqref{PredatorPreyODE} with $\alpha$ and $\beta$ as in Fig.~\ref{fig:Orego}. To obtain Fig.~\ref{fig:EY}, we used the explicit finite difference Euler method and we choose the initial conditions $(u_0, v_0)$ in the following way.
First, we decompose the domain $\Omega=[0,1]^2 = {\Omega_1 \cup \Omega_2}$ with the $EY$-shape set $\Omega_2$ and $|\Omega_2|>0$ sufficiently small. Next, we set  
\begin{nalign}
	\label{ExampleIni}
	u_0 =\overline{ U}_1 \text{ and } v_0=\overline{ V}_1 \text{ on } \Omega_1 \quad \text{as well as} \quad u_0 =0 \text{ and } v_0=\oV_1  \text{ on } \Omega_2.
\end{nalign} 
The graph of the corresponding solution in Fig.~\ref{fig:EY} was obtained for large values of $t>0$ and, by \rev{the} stability results from this work, \rev{this is a good approximation of discontinuous stationary solution}. Here, we observe a graphical \rev{illustration} of inequalities \eqref{DisVar}, namely, $U$ stays close to $\overline{ U}_1 = k_1(\overline{ V}_1)$ on the set $\Omega_1$ and close to $k_3(\overline{ V}_1) = 0$ on the set $\Omega_2$.

\appendix 




\end{document}